\documentclass[final,leqno]{siamltex}

\usepackage{amsfonts,amssymb,amsmath,cmmib57,color,graphics,mathdots}
\usepackage{amstext}
\usepackage{amscd}
\usepackage{graphicx}
\usepackage{longtable}
\usepackage{calc}
\usepackage[latin1]{inputenc}
\usepackage{url}

\newtheorem{remark}{Remark}
\newtheorem{example}{Example}

\newsavebox{\savepar}

\newcommand{\NN}{\mathbb{N}}
\newcommand{\RR}{\mathbb{R}}
\newcommand{\CC}{\mathbb{C}}

\newcommand{\la}{\lambda}

\newcommand{\gln}{{\mbox{{\rm GL}}_{n}(\CC)}}
\newcommand{\glm}{{\mbox{{\rm GL}}_{m}(\CC)}}

\DeclareMathOperator{\orb}{O}

\DeclareMathOperator{\tsp}{T}

\DeclareMathOperator{\tr}{trace}

\DeclareMathOperator{\bun}{B}

\DeclareMathOperator{\pen}{PENCIL}

\DeclareMathOperator{\cod}{cod}

\newcounter{algo}[section]

\title{Generic symmetric matrix pencils\\ with bounded rank}

\author{Fernando De Ter\'{a}n\thanks{Departamento de Matem\'aticas, Universidad Carlos III de Madrid, Avda. Universidad 30, 28911 Legan\'es, Spain. {\tt fteran@math.uc3m.es}, {\tt dopico@math.uc3m.es}}
\and
Andrii Dmytryshyn\thanks{Department of Computing Science, Ume{\aa} University, SE-90187 Ume{\aa}, Sweden. {\tt andrii@cs.umu.se}}
\and
Froil\'{a}n M. Dopico\footnotemark[1]
}

\usepackage[notcite,notref]{showkeys}

\begin{document}


\maketitle

\begin{abstract}
We show that the set of $n \times n$ complex symmetric matrix pencils of rank at most $r$ is the union of the closures of $\lfloor r/2\rfloor +1$ sets of matrix pencils with some, explicitly described, complete eigenstructures. As a consequence, these are the generic complete eigenstructures of $n \times n$ complex symmetric matrix pencils of rank at most $r$. We also show that these closures correspond to the irreducible components of the set of $n\times n$ symmetric matrix pencils with rank at most $r$ when considered as an algebraic set.
\end{abstract}

\begin{keywords}
Matrix pencil, symmetric pencil, strict equivalence, congruence, orbit, bundle, spectral information, complete eigenstructure
\end{keywords}

\begin{AMS}
15A22, 15A18, 15A21, 65F15
\end{AMS}

\section{Introduction}

 The problem of determining the {\em generic} (or most likely) eigenstructures of matrix pencils or, more generally, matrix polynomials, has attracted considerable attention in the past few years. The eigenstructure of a matrix pencil encodes relevant information for the solution of the associated system of differential or differential-algebraic equations (see, for instance, \cite{BeTr12}). Sometimes the matrix pencil (or polynomial) has some particular structure (specially when the pencil or the polynomial comes from real-life applications), which is convenient to preserve in the numerical solution of the corresponding problem. This leads to the problem of determining the generic eigenstructure of the set of matrix pencils (or polynomials) having this special structure. Some of the structures arising in applications are easy to identify just looking at the entries of the pencil or polynomial, like the symmetric, Hermitian, skew-symmetric, palindromic, or alternating structures. Some others are usually hidden, like the low-rank structure, and they are typically considered separately. We are interested in this paper in the generic eigenstructure of low-rank complex symmetric pencils. We want to emphasize that complex symmetric pencils arise in real-life applications, like the analysis of certain cavity resonators \cite{ChASW05} or the design of emitting laser devices \cite{ArChin08}.

The problem of describing the generic eigenstructures is usually solved by providing a decomposition of the set of matrix pencils (or polynomials) of interest (structured or not) as the finite union of the closures of certain open sets, also known as {\it generic sets}. These sets are typically described in terms of their eigenstructures, namely the so-called {\it generic eigenstructures}  \cite{DeTe17,DeDo08,DmDo17,DmDo18}, or in terms of certain parameterizations \cite{DeDL17}.
More precisely, generic eigenstructures for full rank $m \times n$ matrix pencils are presented in \cite{DeEd95,EdEK99,VanD79} and for rank deficient matrix pencils in \cite{DeDo08,DeDL17}. In \cite{DmDo17} the corresponding problem for matrix polynomials is solved. For skew-symmetric matrix pencils and polynomials of odd grades the generic eigenstructures can be found in \cite{DmDo18}. More recently, a description of this kind has been provided in \cite{DeTe17} for $\top$-even, $\top$-odd, and (anti-)palindromic matrix pencils with bounded rank.

A motivation for dealing with low-rank pencils comes from possible applications to the study of the effects of low-rank perturbations on the spectral information of pencils, see e.g., \cite{DeDo07,DeDo16,DeDM08,HoMe94,MeMW17}, including structured low-rank perturbations of pencils with a particular structure, e.g., symmetric, Hermitian, palindromic, skew-symmetric, and alternating pencils \cite{Batz14,Batz15a,MMRR11,MMRR14}.
Also, the study of the generic sets of low-rank matrix pencils generalizes classical results \cite{Wate84} on the algebraic structure of the set of $n\times n$ singular pencils (i.e., those of rank at most $n-1$).

In this paper, we prove that the set of $n \times n$ complex symmetric matrix pencils of (normal) rank at most $r$, denoted by $\pen_{n\times n}^s(r)$, and with $r<n$, is the union of the closures of $\lfloor \frac{r}{2} \rfloor+1$ sets of matrix pencils with certain, explicitly described, complete eigenstructures (where, for a given $q\in\RR$, $\lfloor q\rfloor$ denotes the largest integer which is smaller than or equal to $q$).
In particular, this result illustrates how strong can be the effect of imposing the symmetry on this type of problems, since the set of $n\times n$ general (unstructured) matrix pencils with rank at most $r$ has $r+1$ generic complete eigenstructures  \cite{DeDo08}, and each of the sets of $n\times n$ skew-symmetric, $\top$-even, $\top$-odd, and (anti-)palindromic matrix pencils with rank at most $r$ has a single generic complete eigenstructure \cite{DeTe17,DmDo18}. Thus, the symmetric case presents a completely different and new behavior in the context of generic pencils with bounded rank. Moreover, we emphasize that some of the generic complete eigenstructures for symmetric matrix pencils are not generic for general matrix pencils (see Example~\ref{introex}). The reason for this interesting difference comes from the fact that the inclusion relationships between closure orbits (or bundles) under congruence of symmetric pencils do not necessarily coincide with the inclusion relationships of orbits (or bundles) under strict equivalence. In other words, a symmetric pencil can be the limit of a sequence of (non-symmetric) pencils which are strictly equivalent to another (fixed) symmetric pencil but, however, not being the limit of any sequence of symmetric pencils which are congruent to such pencil. This is illustrated in the following example.

\begin{example}
\label{introex}
Consider the following symmetric matrix pencils
$$
{\cal P}(\la) := \left[\begin{array}{ccc}\la-\la_1\\&\la-\la_2\\&&0\end{array}\right] \text{ with } \la_1\neq\la_2, \quad \text{ and } \quad
{\cal Q}(\la) := \left[\begin{array}{cc|c}&&\la\\&&1\\\hline\la&1\end{array}\right].
$$
There is an arbitrarily small (in norm) non-symmetric perturbation of ${\cal P}(\la)$ such that the resulting matrix pencil is strictly equivalent to ${\cal Q}(\la)$. For example, for arbitrarily small nonzero complex numbers $\epsilon_1$ and $\epsilon_2$, the following strict equivalence holds
$$
\left[\begin{array}{ccc}0&1&\frac{\la_2}{\epsilon_2}\\0&0&\frac{1}{\epsilon_2}\\1&0&0\end{array}\right]
\left[\begin{array}{ccc}\la-\la_1&0&\epsilon_1\\0&\la-\la_2&0\\0&\epsilon_2&0\end{array}\right]
\left[\begin{array}{ccc}1&0&0\\0&0&1\\\frac{\la_1}{\epsilon_1}&\frac{1}{\epsilon_1}&0\end{array}\right]=
\left[\begin{array}{cc|c}0&0&\la\\0&0&1\\\hline\la&1&0\end{array}\right].
$$
In other words, ${\cal P}(\la)$ is in the closure of the strict equivalence orbit of ${\cal Q}(\la)$, see also {\rm\cite{Pokr86}} for a more general result. Note that, despite both ${\cal P}(\la)$ and ${\cal Q}(\la)$ are symmetric, the introduced perturbation with $\epsilon_1$ and $\epsilon_2$ is not symmetric. Moreover, the result of this paper states that there is no any arbitrarily small symmetric perturbation of ${\cal P}(\la)$ which is strictly equivalent or, equivalently, congruent to ${\cal Q}(\la)$. (Recall that two complex symmetric matrix pencils are strictly equivalent if and only if they are congruent, see, for instance, {\rm\cite[p. 339]{Thom91}} for a proof.) Namely, ${\cal P}(\la)$ is not in the closure of the congruence orbit of ${\cal Q}(\la)$.
\end{example}

The pencil ${\cal P}(\la)$ in Example \ref{introex} corresponds to a generic complete eigenstructure for symmetric $3\times 3$ singular pencils, consisting of two simple eigenvalues, one left minimal index equal to zero, and another right minimal index equal to zero (see Theorem \ref{th:improvedDeDo}). However, for general (unstructured) $3\times3$ singular pencils, all generic eigenstructures do not have eigenvalues at all (see \cite{Wate84} and \cite[Cor. 7.2]{DeEd95}), and, therefore, ${\cal P}(\la)$ does not have one of the generic eigenstructure for $3\times3$ singular pencils under strict equivalence. By contrast, ${\cal Q}(\la)$ is a generic eigenstructure for both general (i.e., under strict equivalence) and symmetric (i.e., under congruence) singular $3\times 3$ matrix pencils.

We want to mention that the limit case $r=n$ has been omitted from the analysis carried out in this paper for obvious reasons. More precisely, the generic eigenstructure of $\pen_{n\times n}^s(n)$ consists of $n$ simple eigenvalues. This is because pencils in $\pen_{n\times n}^s(n)$ are generically regular, and for regular symmetric pencils the eigenvalues are generically simple, as happens with general (non-structured) pencils. It is, however, interesting to note the contrast between the case $r<n$ (there are $\lfloor \frac{r}{2} \rfloor+1$ different generic eigenstructures in $\pen_{n\times n}^s(r)$) and the case $r=n$ (there is just one generic eigenstructure in $\pen_{n\times n}^s(n)$).

The rest of the paper is organized as follows. Section \ref{sec:basic} includes some basic results on general and symmetric pencils, which allow us to describe, in Section~\ref{sec:main}, the set $\pen_{n\times n}^s(r)$ as the union of the closures of certain matrix pencil bundles. Section~\ref{sec:irredcomp} is devoted to prove that the closures of these generic bundles are the irreducible components of $\pen_{n\times n}^s(r)$. In Section \ref{sec:codim} the codimensions of the generic bundles for sets of symmetric matrix pencils of bounded rank are computed. Finally, some possible directions of a future research that may build up on the result of this paper are discussed in Section \ref{sec:future}.

The reader should bear in mind throughout this paper that all the matrices that we consider have complex entries.

\section{Basic notation, definitions, and previous results}
\label{sec:basic}
In this paper all matrices are denoted by capital letters, e.g. $A, \ B, \ U$, and matrix pencils are denoted by calligraphic capital letters, e.g., ${\cal P }, \ {\cal L}, \ {\cal M}$.
We will use either $\diag(A_1,\hdots,A_k)$, or $A_1\oplus \cdots\oplus A_k$, or even $\bigoplus_{i=1}^k A_i$, to denote a matrix or a matrix pencil which is block diagonal with diagonal blocks equal to $A_1,\hdots,A_k$. The transpose and the conjugate transpose of a matrix $M$ are denoted by $M^\top$ and $M^*$, respectively. The transpose of a pencil ${\cal P}(\la)=\la A+B$ is the pencil ${\cal P}^\top(\la):= {\cal P}(\la)^\top=\la A^\top+B^\top$ (we will use the notation ${\cal P}(\la)^\top$, since it is more intuitive). A {\em symmetric pencil} is a pencil ${\cal P}(\la)$ such that ${\cal P}(\la)^\top={\cal P}(\la)$ (i.e., its coefficients are symmetric matrices).

The {\em rank} of a matrix pencil ${\cal P }(\la)$, denoted by $\rank {\cal P }$ (omitting the dependence on $\la$), is the size of the largest non-identically zero minor of ${\cal P }(\la)$ (in other words, the rank of ${\cal P }(\la)$ when considered as a matrix in the field of rational functions in $\la$). This notion appear sometimes in the literature under the name {\em normal rank} (see, for instance, \cite{EdEK99}).

We denote by $\pen_{m \times n}$ the set of all matrix pencils with size $m\times n$.

We define the {\it orbit} of a matrix pencil $\lambda A + B$ under the
action of the group $\glm\times \gln$ on the space $\pen_{m\times n}$ by strict equivalence as follows:
\begin{equation} \label{equorbit}
\orb^e (\lambda A + B) := \{U^{-1} (\lambda A + B) V \ : \ U \in \glm, V \in \gln\}.
\end{equation}

Similarly to the general case, we denote the vector space of symmetric matrix pencils of size $n\times n$ by $\pen_{n\times n}^{s}$.
Let us define the {\it orbit} of $\lambda A + B$ under the
action of the group $\gln$ on $\pen_{n\times n}^{s}$ by congruence as:
\begin{equation} \label{congorbit}
\orb^c (\lambda A + B): = \{W^{\top} (\lambda A + B) W \ : \ W \in \gln\}.
\end{equation}
If we identify $\pen_{n\times n}$ with $\CC^{2n^2}$, then $\pen_{n\times n}^{s}$ becomes a subspace of $\CC^{2n^2}$, and we consider in $\pen_{n\times n}^{s}$ the subspace topology induced by the Euclidean topology in $\CC^{2n^2}$. The closure of a set $S\subseteq \pen_{n \times n}^{s}$ in such topology will be denoted by $\overline S$. In those cases where $S$ depends on some parameters or variables $\mathbf p$, like for the orbits or bundles, we will write $\overline S(\mathbf p)$ instead of $\overline{S(\mathbf p)}$.


From the matrices:
\begin{equation*}
F_d :=
\begin{bmatrix}
0&1&&\\
&\ddots&\ddots&\\
&&0&1\\
\end{bmatrix}_{d\times(d+1)} \qquad\mbox{\rm and}\qquad
G_d :=
\begin{bmatrix}
1&0&&\\
&\ddots&\ddots&\\
&&1&0\\
\end{bmatrix}_{d\times(d+1)},
\end{equation*}
we build up the following pencils
$$
{\cal L}_d(\la):=\la G_d+F_d,\qquad
{\cal M}_d(\la):=
\begin{bmatrix}0&{\cal L}_d(\la)^\top\\
{\cal L}_d(\la)&0
\end{bmatrix}_{(2d+1)\times(2d+1)},
$$
which are mainly used in Section \ref{sec:main}. The matrix ${\cal M}_0(\la)$ corresponds to a $1\times 1$ null matrix. This matrix is a degenerate case of ${\cal M}_d(\la)$, which is obtained after collapsing ${\cal L}_0(\la)$ and ${\cal L}_0(\la)^\top$, namely a null column and a null row, respectively.

As we are going to see (Theorem \ref{canonicalform_th}), the spectral information of a symmetric matrix pencil is encoded in a direct sum of blocks of the following types (though they are all matrix pencils, we omit the dependence on the variable $\la$ for brevity):

\begin{itemize}
\item Symmetric blocks associated with a couple of a right and a left minimal index:
$$
{\cal M}_d.
$$
\item Symmetric Jordan-like blocks associated with finite eigenvalues:
$$
{\cal J}_\ell^s(\mu):=\left[\begin{array}{cccc}&&1&\la-\mu\\&\iddots&\iddots&\\1&\la-\mu&&\\\la-\mu&&&\end{array}\right]_{\ell\times\ell}\quad (\mu\in\CC).
$$
\item Symmetric Jordan-like blocks associated with the infinite eigenvalue:
$$
{\cal J}_k^s(\infty):=\left[\begin{array}{cccc}&&\la&1\\&\iddots&\iddots&\\\la&1&&\\1&&&\end{array}\right]_{k\times k}.
$$
\end{itemize}
Note that ${\cal J}_1^s(\mu)=\la-\mu$, and ${\cal J}_1^s(\infty)=1$.

 The complete eigenstructure of a matrix pencil is the set of invariants of the pencil under strict equivalence, consisting of the lists of {\em elementary divisors} and {\em left} and {\em right minimal indices} \cite{VaDe83}. In other words, it is the underlying structure which is common to all pencils in the same strict equivalence orbit $\orb^{e}(\la A+B)$. This eigenstructure is shown in the {\em Kronecker canonical form} of $\la A+B$, which is a particular pencil in $\orb^{e}(\la A+B)$ displaying its complete eigenstructure \cite[Ch. XII]{Gant59}. If $\la A+B$ is a symmetric pencil, a key fact is that its complete eigenstructure can be also recovered by means of congruence transformations. In particular, there is a canonical form for congruence of symmetric pencils displaying their complete eigenstructure (namely, a Kronecker-like canonical form). This is stated in the following result, which is one of the basic tools in our developments.

\begin{theorem}\label{canonicalform_th}{\rm (Canonical form under congruence of symmetric pencils, \cite{Thom91})}. Every symmetric pencil is congruent to a block diagonal pencil whose diagonal blocks are of the forms ${\cal M}_d,{\cal J}^s_\ell(\mu)$, and ${\cal J}^s_k(\infty)$. The number of blocks associated with each eigenvalue $\mu,\infty$, together with their sizes, is unique.
\end{theorem}

The {\em regular part} of a symmetric matrix pencil $\lambda A + B$ consists of the blocks ${\cal J}^s_\ell(\mu)$ and ${\cal J}^s_k(\infty)$ corresponding to the finite and infinite eigenvalues, respectively. The {\em singular part} of $\lambda A + B$ consists of the blocks ${\cal M}_ d$ corresponding to the right (column) and left (row) minimal indices, which are equal to each other in the case of symmetric matrix pencils. We refer the reader to \cite[Th. 2.17]{Batz15b} for a representation of the canonical form close to the one we present in Theorem \ref{canonicalform_th}.

For a symmetric matrix pencil $\lambda A + B$, define the {\it congruence bundle} $\bun^{\rm{c}}(\lambda A + B)$ to be the union of symmetric matrix pencil orbits under congruence with the same canonical block structure (equal block sizes) but where the eigenvalues are unspecified.
This definition is analogous to the definition of bundle for matrix pencils and polynomials
\cite{Dmyt17,DmJK17,DmKa14,EdEK99,JoKV13}.

A sequence of integers $\eta=(n_1, n_2, n_3, \dots)$ such that
$n_1+n_2+n_3 + \dots =n$ and $n_1\ge n_2 \ge \dots \ge 0$ is called
an {\it integer partition} of $n$, see e.g., \cite{EdEK99} for more details and references. For any $a \in \mathbb Z_{\ge 0}$ we
define $\eta+a$ as $(n_1+a, n_2+a, n_3+a, \dots)$.
The set of all integer partitions form a partially ordered set with
respect to { the following order:}
$\eta \preceq \nu $, where $\nu=(m_1, m_2, m_3, \dots)$, if and only if
$n_1+n_2+ \dots + n_i \le m_1 +m_2 + \dots + m_i,$ for $i\ge 1.$ When
$\eta \preceq \nu $ and
$\eta \neq \nu$ then we write
$\eta \prec \nu $.

The complete eigenstructure of a symmetric matrix pencil
  ${\mathcal S}$, with $e$ distinct eigenvalues $\mu_j \in \mathbb C\cup\{\infty\}$,
  can now be represented by the set of {\it integer
    partitions} $\varepsilon_{\mathcal S}$ and $\{ \delta^{\mu_j}_{\mathcal S}:\ j = 1, \dots ,  e \}$ (so called {\em Weyr-type characteristics}) as follows:
$\varepsilon_{\mathcal S}:=(r_0,r_1, \dots )$, where $r_k$ is the number of ${\cal M}_d$ blocks with the index $d$ greater than or equal to $k$;
for each distinct $\mu_j \in {\mathbb C}\cup\{\infty\}$, $\delta^{\mu_j}_{\mathcal S}:=
  (h_1^{\mu_j}, h_2^{\mu_j}, \dots)$, where $h_k^{\mu_j}$ is the
  number of Jordan blocks ${\cal J}^s_\ell(\mu_j)$ of size $\ell$
  greater than or equal to $k$.

\section{Main results}\label{sec:main}

We start with a block (anti-triangular) decomposition of symmetric pencils that separates the regular part from the singular part corresponding to the right minimal indices, and the one corresponding to the left minimal indices. It can be achieved by unitary congruence, and it is the analogue for symmetric pencils to the staircase form of general pencils introduced in \cite[p. 133]{Van79}.

\begin{theorem}\label{antitriangular_th}{\rm (Block anti-triangular form of symmetric pencils).} Let ${\cal S}(\la)$ be a symmetric pencil. Then, there is a unitary matrix $Q$ such that
\begin{equation}\label{antitriangular}
{\cal S}(\la)=Q^\top\left[\begin{array}{ccc}{\cal A}(\la)&{\cal B}(\la)&{\cal S}_{\mbox{\rm\tiny right}}(\la)\\
{\cal B}(\la)^\top&{\cal S}_{\mbox{\rm\tiny reg}}(\la)&0\\{\cal S}_{\mbox{\rm\tiny right}}(\la)^\top&0&0\end{array}\right]Q,
\end{equation}
where:
\begin{itemize}
\item[\rm(i)] ${\cal A}(\la)$ is a symmetric pencil.

\item[\rm(ii)] ${\cal S}_{\mbox{\rm\tiny reg}}(\la)$ is a regular symmetric pencil whose elementary divisors are exactly those of ${\cal S}(\la)$.

\item[\rm(iii)] ${\cal S}_{\mbox{\rm\tiny right}}(\la)$ is a pencil whose complete eigenstructure consists only of the right minimal indices of ${\cal S}(\la)$.
\end{itemize}
As a consequence, ${\cal S}_{\mbox{\rm\tiny right}}(\la)^\top$ is a pencil whose complete eigenstructure consists only of the left minimal indices of ${\cal S}(\la)$.
\end{theorem}
\begin{proof}
By Theorem \ref{canonicalform_th}, every symmetric pencil ${\cal S}(\la)$ is congruent to a direct sum of blocks of the form ${\cal M}_d,{\cal J}^s_\ell(\mu)$, and ${\cal J}^s_k(\infty)$. Define ${\cal J}(\la)$ to be a direct sum of all the symmetric Jordan-like blocks ${\cal J}^s_\ell(\mu)$ and ${\cal J}^s_k(\infty)$. Now, assume that there is some block ${\cal M}_d$ with $d\neq0$ (see the last paragraph of the proof for the other case). Splitting each block ${\cal M}_d$ and permuting the rows and columns in the canonical form of ${\cal S}(\la)$ we have
\begin{equation}\label{santidiag}
{\cal S}(\la)=
W^\top \begin{bmatrix}
&&&&&&{\cal L}_{d_t}(\la)\\
&&&&&\reflectbox{$\ddots$}&\\
&&&&{\cal L}_{d_1}(\la)&&\\
&&&{\cal J}(\la)&&&\\
&&{\cal L}_{d_1}(\la)^\top&&&&\\
&\reflectbox{$\ddots$}&&&&&\\
{\cal L}_{d_t}(\la)^\top&&&&&&\\
\end{bmatrix} W,
\end{equation}
where $W$ is nonsingular. For the blocks ${\cal M}_0$, the corresponding split block ${\cal L}_0(\la)$ in \eqref{santidiag} consists of a zero column, and ${\cal L}_0(\la)^\top$ consists of a zero row. Write $W^\top=Q^\top R$  (equivalently, $W=R^\top Q$), where $Q^\top$ is unitary and $R$ is upper-triangular, see e.g. \cite[p.112, Theorem 2.6.1]{HoJo85}. Define
\begin{equation*}
{\cal L}(\la):=
\begin{bmatrix}
&&{\cal L}_{d_t}(\la)\\
&\reflectbox{$\ddots$}&\\
{\cal L}_{d_1}(\la)&&\\
\end{bmatrix}.
\end{equation*}
Now, by an adequate conformal partition, we have
\begin{equation}\label{partition}
{\cal S}(\la)=
Q^\top
\begin{bmatrix}
R_{11}&R_{12}&R_{13}\\
0&R_{22}&R_{23}\\
0&0&R_{33}\\
\end{bmatrix}
\begin{bmatrix}
0&0&{\cal L}(\la)\\
0&{\cal J}(\la)&0\\
{\cal L}(\la)^\top&0&0\\
\end{bmatrix}
\begin{bmatrix}
R_{11}^\top&0&0\\
R_{12}^\top&R_{22}^\top&0\\
R_{13}^\top&R_{23}^\top&R_{33}^\top\\
\end{bmatrix}
Q.
\end{equation}
Multiplying the matrix pencils in the middle of the right-hand side of \eqref{partition} 
we get the following symmetric pencil
\begin{equation*}
\begin{split}
&\normalsize{{\cal Z}(\la):=} \\
&
{\footnotesize\begin{bmatrix}
R_{13}{\cal L}(\la)^\top R_{11}^\top+R_{12}{\cal J}(\la)R_{12}^\top+R_{11}{\cal L}(\la)R_{13}^\top&R_{12}{\cal J}(\la)R_{22}^\top+R_{11}{\cal L}(\la)R_{23}^\top&R_{11}{\cal L}(\la)R_{33}^\top\\
R_{22}{\cal J}(\la)R_{12}^\top+R_{23}{\cal L}(\la)^\top R_{11}^\top&R_{22}{\cal J}(\la)R_{22}^\top&0\\
R_{33}{\cal L}(\la)^\top R_{11}^\top&0&0\\
\end{bmatrix}}.
\end{split}
\end{equation*}
\noindent Note that $R_{11}{\cal L}(\la)R_{33}^\top$ is strictly equivalent to ${\cal L}(\la)$ and thus it is a pencil ${\cal S}_{\mbox{\rm\tiny right}}(\la)$ whose complete eigenstructure consists only of the right minimal indices of ${\cal S}(\la)$. As a consequence, ${\cal S}_{\mbox{\rm\tiny right}}(\la)^\top=R_{33}{\cal L}(\la)^\top R_{11}^\top$ is a pencil whose complete eigenstructure consists only of the left minimal indices of ${\cal S}(\la)$, and similarly $R_{22} {\cal J}(\la)R_{22}^\top$ is a regular symmetric pencil whose complete eigenstructure consists of the elementary divisors of ${\cal S}(\la)$. Therefore ${\cal Z}(\la)$ has the block anti-triangular structure of the block-pencil in \eqref{antitriangular}. Thus ${\cal S}(\la)=Q^\top{\cal Z}(\la)Q$ is the desired decomposition.

In the case where $d=0$, for all singular blocks ${\cal M}_d$ (namely, when all minimal indices of ${\cal S}(\la)$ are equal to zero), the blocks ${\cal L}_{d_1}(\la),\hdots,{\cal L}_{d_t}(\la)$ and ${\cal L}_{d_1}(\la)^\top,\hdots,{\cal L}_{d_t}(\la)^\top$ do not appear in the central matrix of the right-hand side of \eqref{santidiag}. In this case, this matrix is of the form
\begin{equation}\label{santidiag2}
\left[\begin{array}{cc}{\cal J}(\la)&0\\0&0\end{array}\right],
\end{equation}
where the last zero columns (respectively, rows) correspond to the right (resp., left) minimal indices of ${\cal S}(\la)$. Then, the matrix ${\cal L}(\la)$ does not appear in \eqref{partition}, and the central matrix in the right-hand side of this equation is equal to \eqref{santidiag2}. The previous proof can be adapted to this case just by removing also the blocks $R_{11},R_{12},$ and $R_{13}$, together with their corresponding transposes.

The case when ${\cal S} (\la)$ is regular, i.e., it does not have minimal indices at all, is trivial, since in this case ${\cal S}(\la)=W^\top {\cal J}(\la)W=Q^\top{\cal S}_{	\mbox{\rm \tiny reg}}(\la)Q$.
\end{proof}

\begin{remark}
As it can be seen from the proof of Theorem {\rm\ref{antitriangular_th}}, it may happen that ${\cal S}_{\mbox{\rm\tiny right}}(\la)$ (and, as a consequence, ${\cal S}_{\mbox{\rm\tiny right}}(\la)^\top$ as well) is not present in \eqref{antitriangular}. This happens when all minimal indices of ${\cal S}(\la)$ are equal to zero or when ${\cal S}(\la)$ is regular. In this case, \eqref{antitriangular} reads
$$
{\cal S}(\la)=Q^\top\left[\begin{array}{cc}{\cal S}_{\mbox{\rm\tiny reg}}(\la)&0\\0&0\end{array}\right]Q,
$$
and the presence of the right (respectively, left) null minimal indices is displayed in the last null columns (resp., rows) of the pencil in the middle of the right-hand side. In case such null columns and rows are not present, ${\cal S}(\la)$ is regular and does not have minimal indices.
\end{remark}

 The main result of this paper is Theorem \ref{th:improvedDeDo}, which provides a description of the generic complete eigenstructures of $\pen_{n\times n}^s(r)$, for $r<n$. In order to do that, we provide a decomposition of $\pen_{n\times n}^s(r)$ as the union of the closures of $\lfloor\frac{r}{2}\rfloor+1$ symmetric bundles, which determine the generic eigenstructures.

\begin{theorem}{\rm (Generic eigenstructures of symmetric matrix pencils with bounded rank).} \label{th:improvedDeDo}
Let $n$ and $r$ be integers such that $n \geq 2$ and $1\leq r \leq n-1$.
Define the following $\lfloor\frac{r}{2}\rfloor+1$ symmetric canonical forms of $n\times n$ complex symmetric matrix pencils with rank $r$:
\begin{equation}\label{max}
{\cal
K}_{a} (\lambda):=\diag(\underbrace{{\cal M}_{\alpha+1},\hdots,{\cal M}_{\alpha+1}}_{s},
\underbrace{{\cal M}_{\alpha},\hdots,{\cal M}_{\alpha}}_{n-r-s}, {\cal J}_1(\mu_1), \dots , {\cal J}_1(\mu_{r-2a}) ),\,
\end{equation}
for $a=0,1,\ldots,\lfloor\frac{r}{2}\rfloor\, $, where $a=(n-r)\alpha+s$ is the Euclidean division of $a$ by $n-r$, and $\mu_1,\hdots,\mu_{r-2a}$ are arbitrary complex numbers (different from each other).
Then:
\begin{enumerate}
\item[\rm (i)] For every $n\times n$ symmetric pencil ${\cal S }(\lambda)$ with rank
at most $r$, there exists an integer $a$ such that
$\overline{\bun^c}({\cal K}_{a}) \supseteq\overline{\bun^c}({\cal S})$.
\item[\rm (ii)] $\overline{\bun^c}({\cal K}_{a})
\not\supseteq \overline{\bun^c}({\cal K}_{a'})$ whenever $a \ne
a'$.
\item[\rm (iii)] The set $\pen_{n\times n}^s(r)$ is a closed subset of $\pen^{s}_{n \times n}$, and it is equal to $\displaystyle \bigcup_{0\leq a \leq \lfloor\frac{r}{2}\rfloor} \overline{\bun^c}({\cal K}_{a} )$.
\end{enumerate}
\end{theorem}
\begin{proof} The set $\pen_{n\times n}^s(r)$ is a closed subset of $\pen^{s}_{n \times n}$ because it is the intersection of $\pen^{s}_{n \times n}$ with the set of $n\times n$ pencils with rank at most $r$, which is a closed set. Therefore, claim (iii) is an immediate consequence of (i), so we only need to prove (i) and (ii).

Let us first prove (i). First, we are going to see that we can reduce the proof to the set of pencils with rank exactly $r$. Assume that any symmetric pencil with rank exactly $r$ belongs to $\displaystyle \bigcup_{0\leq a \leq  \lfloor\frac{r}{2}\rfloor} \overline{\bun^c}({\cal K}_{a} )$. We are going to show that any symmetric pencil ${\cal S}(\la)$ with rank $r_0\leq r$ is the limit of a sequence of symmetric pencils with rank $r$. This will imply that any symmetric pencil ${\cal S}(\la)$ with $\rank {\cal S}=r_0<r$ is in the closure of the set of symmetric pencils with rank $r$, which is contained in $\displaystyle \bigcup_{0\leq a \leq \lfloor\frac{r}{2}\rfloor} \overline{\bun^c}({\cal K}_{a} )$, because this last set is closed. Therefore, ${\cal S}(\la)$ will belong to $\displaystyle \bigcup_{0\leq a \leq \lfloor\frac{r}{2}\rfloor} \overline{\bun^c}({\cal K}_{a} )$ as well.

More precisely, we are going to see that, if ${\cal S}(\la)$ is a symmetric pencil with $\rank {\cal S}=r_0<n$, then it is the limit of a sequence of symmetric pencils with rank $r_0+1$. By Theorem \ref{canonicalform_th}, ${\cal S}(\la)$ is congruent to a pencil of the form
$$
{\cal K}_{\cal S}(\la):=\diag({\cal M}_{d_1},\widetilde {\cal S}(\la)),
$$
for some $d_1\geq0$, with $\widetilde {\cal S}$ being a symmetric pencil which is not relevant in our argument. Now, ${\cal K}_{\cal S}(\la)$ is the limit of the sequence $\{{\cal S}_m(\la)\}_{m\in\NN}$, where ${\cal S}_m(\la):={\cal K}_{\cal S}(\la)+(1/m)E_{11}$, and $E_{11}$ is the $n\times n$ matrix having an entry equal to $1$ in the $(1,1)$ position and zeroes elsewhere. Note that ${\cal S}_m(\la)$ is a symmetric pencil with $\rank {\cal S}_m=r_0+1$, for all $m\in\NN$, so this proves the claim.

Now, let ${\cal S}(\la)$ be an $n\times n$ symmetric matrix pencil with $ \rank {\cal S}=r$. By Theorem~\ref{canonicalform_th}, ${\cal S}(\la)$ is congruent to a direct sum of blocks of the forms ${\cal M}_d,{\cal J}^s_\ell(\mu)$, and ${\cal J}^s_k(\infty)$, and the number of blocks of type ${\cal M}_d$ is equal to $n-r$, by the rank-nullity theorem. Then, claim (i) follows from the following three facts:
\smallskip

\begin{itemize}

\item[(a)] ${\cal J}_\ell^s(\mu)\in\overline{\bun^c}({\cal J}_1(\mu_1),\hdots,{\cal J}_1(\mu_\ell))$.

\item[(b)] ${\cal J}_k^s(\infty)\in\overline{\bun^c}({\cal J}_1(\mu_1),\hdots,{\cal J}_1(\mu_k))$.

\item[(c)] $\diag({\cal M}_{d_1},\hdots,{\cal M}_{d_{n-r}})\in\overline{\orb^c}(\diag(\underbrace{{\cal M}_{\alpha+1},\hdots,{\cal M}_{\alpha+1}}_{\mbox{\tiny $s$}},\underbrace{{\cal M}_{\alpha},\hdots,{\cal M}_{\alpha})}_{\mbox{\tiny $n-r-s$}}),$ with \break$\sum_{i=1}^{n-r} d_i=(n-r)\alpha+s$ being the Euclidean division of $\sum_{i=1}^{n-r}d_i$ by $n-r$.
\end{itemize}

\smallskip

Let us prove (a)--(c). For claim (a) just note that ${\cal J}_\ell^s(\mu)$ is the limit of the following sequence of pencils (as $m$ tends to infinity):
$$
\left[\begin{array}{cccccc}&&&1&\la-\mu\\&&1&\la-\mu&\\&\iddots&\iddots&&\\1&\la-\mu&\\\la-\mu&&&&1/m\end{array}\right]_{\ell\times\ell},
$$
and that these pencils belong to ${\bun}^c({\cal J}_1(\mu_1),\hdots,{\cal J}_1(\mu_\ell))$.

Claim (b) can be proved in a similar way noticing that ${\cal J}_k^s(\infty)$ is the limit of the sequence of pencils
$$
\left[\begin{array}{cccccc}&&&\la&1\\&&\la&1&\\&\iddots&\iddots&&\\\la&1&\\1&&&&\la/m\end{array}\right]_{k\times k},
$$
and that these pencils belong, again, to ${\bun}^c({\cal J}_1(\mu_1),\hdots,{\cal J}_1(\mu_k))$.

As for claim (c), it follows from the general covering rules for (strict) equivalence orbits of matrix pencils. More precisely, it is known that, if $d_1$ and $d_2$ are such that $d_1>d_2$, then ${\cal L}_{d_1}\oplus {\cal L}_{d_2}\in\overline {\orb^c}({\cal L}_{d_1-1}\oplus {\cal L}_{d_2+1})$ (see \cite[p. 110--111]{Pokr86}). Then, there are two sequences of nonsingular matrices $\{P_m\}_{m\in\NN}$ and $\{Q_m\}_{m\in\NN}$ such that $P_m({\cal L}_{d_1-1}\oplus {\cal L}_{d_2+1}) Q_m$ converges to ${\cal L}_{d_1}\oplus {\cal L}_{d_2}$. Therefore, the sequence of pencils
\begin{equation}\label{sequence}
\left[\begin{array}{cc} Q_m^\top&0\\0&P_m\end{array}\right]\left[\begin{array}{cc}0&{\cal L}_{d_1-1}^\top\oplus {\cal L}_{d_2+1}^\top\\{\cal L}_{d_1-1}\oplus {\cal L}_{d_2+1}&0\end{array}\right]\left[\begin{array}{cc}Q_m&0\\0&P_m^\top\end{array}\right]
\end{equation}
converges to $\left[\begin{smallmatrix}0&{\cal L}_{d_1}^\top\oplus {\cal L}_{d_2}^\top\\{\cal L}_{d_1}\oplus {\cal L}_{d_2}&0\end{smallmatrix}\right]$. But note that the pencils in \eqref{sequence} are all congruent to $\left[\begin{smallmatrix}0&{\cal L}_{d_1-1}^\top\oplus {\cal L}_{d_2+1}^\top\\{\cal L}_{d_1-1}\oplus {\cal L}_{d_2+1}&0\end{smallmatrix}\right]$, and that the pencil $\left[\begin{smallmatrix}0&{\cal L}_{d_1}^\top\oplus {\cal L}_{d_2}^\top\\{\cal L}_{d_1}\oplus {\cal L}_{d_2}&0\end{smallmatrix}\right]$ is symmetric and with the same spectral structure as ${\cal M}_{d_1}\oplus {\cal M}_{d_2}$. As a consequence, ${\cal M}_{d_1}\oplus{\cal M}_{d_2}\in\overline{\orb^c}({\cal M}_{d_1-1}\oplus {\cal M}_{d_2+1})$. By repeating this argument with a direct sum $\bigoplus_{i=1}^{n-r} {\cal M}_{d_i}$, as long as there are two indices $d_i,d_j$ such that $|d_i-d_j|\geq2$, we end up with a pencil as in the right hand side of the inclusion in claim (c).

Now, let us prove (ii). First, we are going to see that if $a'>a$ then $\bun^c({\cal K}_{a'})$ is not contained in  $\overline{\bun^c}({\cal K}_{a})$. By contradiction, let us assume that ${\cal K}_{a'}(\la)\in\overline{\bun^c}({\cal K}_{a})$, for some choice of the eigenvalues $\mu_1,\hdots,\mu_{r-2a'}$ in \eqref{max}. This means that there is a sequence ${\cal S}_m(\la)\in\bun^c({\cal K}_a)$ such that ${\cal S}_m(\la)$ converges to ${\cal S}(\la):={\cal K}_{a'}(\la)$. Then (see, for instance, \cite[Lemma 1.2]{Hoyo90}) there is a majorization of the Weyr sequence of right minimal indices
\begin{equation}\label{majorization}
\varepsilon_{\cal S}\prec \varepsilon_{{\cal S}_m}.
\end{equation}
Notice that
$$
\begin{array}{ccc}
\varepsilon_{\cal S}&=&(\underbrace{n-r,\hdots,n-r}_{\alpha'+1},s',0,0,\hdots),\\
\varepsilon_{{\cal S}_m}&=&(\underbrace{n-r,\hdots,n-r}_{\alpha+1},s,0,0,\hdots),
\end{array}
$$
where
$$
\begin{array}{cccc}
a&=&(n-r)\alpha+s,&0\leq s<n-r,\\
a'&=&(n-r)\alpha'+s',&0\leq s'<n-r,
\end{array}
$$
are the Euclidean division of, respectively, $a$ and $a'$ by $n-r$. Since $a'>a$, it follows that either $\alpha'>\alpha$ or $\alpha'=\alpha$ and $s'>s$, in contradiction with \eqref{majorization}.

It remains to prove that if $a'<a$ then $\bun^c({\cal K}_{a'})$ is not contained in  $\overline{\bun^c}({\cal K}_{a})$ either. By contradiction, if $\bun^c({\cal K}_{a'})\subseteq\overline{\bun^c}({\cal K}_{a})$, then any pencil congruent to ${\cal K}_{a'}(\la)$ as in \eqref{max}, with $\mu_i\neq\mu_j$, for $i\neq j$, must be the limit of a sequence of pencils in $\bun^c({\cal K}_{a})$. Let $\{{\cal S}_m(\la)\}_{m\in\NN}$ be a sequence of pencils with ${\cal S}_m(\la)\in\bun^c({\cal K}_a)$, for all $m\in\NN$. Then, by Theorem \ref{antitriangular_th},
\begin{equation}\label{Sm}
{\cal S}_m(\la)=Q_m^\top\left[\begin{array}{ccc}{\cal A}_m(\la)&{\cal B}_m(\la)&{\cal S}_{\tiny\mbox{\rm right}}^{(m)}(\la)\\{\cal B}_m(\la)^\top&S_{\tiny\mbox{\rm reg}}^{(m)}(\la)&0\\{\cal S}_{\tiny\mbox{\rm right}}^{(m)}(\la)^\top&0&0\end{array}\right]Q_m,
\end{equation}
with $Q_m\in\CC^{n\times n}$ being a unitary matrix, for all $m\in\NN$, and
\begin{itemize}
\item ${\cal S}_{\tiny\mbox{\rm right}}^{(m)}(\la)$ has size $a\times (n-r+a)$, and complete eigenstructure consisting of the right minimal indices of ${\cal K}_a(\la)$,
\item ${\cal S}_{\tiny\mbox{\rm right}}^{(m)}(\la)^\top$ has size $(n-r+a)\times a$, and complete eigenstructure consisting of the left minimal indices of ${\cal K}_a(\la)$,
\item ${\cal S}^{(m)}_{\tiny\mbox{\rm reg}}(\la)$ is a regular pencil of size $(r-2a)\times(r-2a)$.
\end{itemize}

Now, let us assume that ${\cal S}_m(\la)$ converges to some pencil ${\cal S}(\la)$. Since the set of unitary $n\times n$ matrices is a compact subset of the metric space $(\CC^{n\times n},\|\cdot\|_2)$, the sequence $\{Q_m\}_{m\in\NN}$ has a convergent subsequence, say $\{Q_{m_j}\}_{j\in\NN}$, whose limit is a unitary matrix (see, for instance, \cite[Th. 2.36]{Knap05}). Set
$$
{\cal H}_m(\la):=\left[\begin{array}{ccc}{\cal A}_m(\la)&{\cal B}_m(\la)&{\cal S}_{\tiny\mbox{\rm right}}^{(m)}(\la)\\{\cal B}_m(\la)^\top&{\cal S}_{\tiny\mbox{\rm reg}}^{(m)}(\la)&0\\{\cal S}_{\tiny\mbox{\rm right}}^{(m)}(\la)^\top&0&0\end{array}\right]
$$
for the matrix in the middle of the right-hand side in \eqref{Sm}. Then the sequence $\{{\cal H}_{m_j}(\la)\}_{j\in\NN}$ is convergent as well, since ${\cal H}_{m_j}(\la)=\overline Q_{m_j}{\cal S}_{m_j}(\la)Q_{m_j}^*$, and both $\{Q_{m_j}\}_{j\in\NN}$ (and, as a consequence, $\{\overline Q_{m_j}\}_{j\in\NN}$ and $\{Q_{m_j}^*\}_{j\in\NN}$) and $\{{\cal S}_{m_j}(\la)\}_{j\in\NN}$ are convergent, because any subsequence of ${\cal S}_m(\la)$ converges to its limit (here $\overline Q_{m_j}$ denotes the matrix whose entries are the complex conjugates of the entries of $Q_{m_j}$). Moreover, by continuity of the zero blocks in the block-structure, $\{{\cal H}_{m_j}(\la)\}_{j\in\NN}$ converges to a matrix pencil of the form
\begin{equation}\label{h}
{\cal H}(\la)=\left[\begin{array}{ccc}{\cal A}(\la)&{\cal B}(\la)&{\cal C}(\la)\\{\cal B}(\la)^\top&{\cal R}(\la)&0\\{\cal C}(\la)^\top&0&0\end{array}\right],
\end{equation}
with
\begin{itemize}
\item ${\cal C}(\la)$ being of size $a\times (n-r+a)$,

\item ${\cal C}(\la)^\top$ being of size $(n-r+a)\times a$, and

\item ${\cal R}(\la)$ being of size $(r-2a)\times(r-2a)$.
\end{itemize}

Therefore, the sequence $\{{\cal S}_{m_j}(\la)\}_{j\in\NN}$ converges to $Q^\top{\cal H}(\la)Q$, where $Q:=\displaystyle\lim_{j\rightarrow\infty}Q_{m_j}$ is unitary. Since $\{{\cal S}_m(\la)\}_{m\in\NN}$ is convergent, any subsequence must converge to its limit, so $\displaystyle\lim_{m\rightarrow\infty}{\cal S}_m(\la)={\cal S}(\la)=Q^\top{\cal H}(\la)Q$.

Now, let us assume that
\begin{eqnarray}
{\cal S}(\la)\in\bun^c({\cal K}_{a'}),\qquad \mbox{\rm with $a'<a$}\label{buncond}, \qquad\mbox{\rm and}\\
\mbox{\rm ${\cal S}(\la)$ has $r-2a'$ simple eigenvalues}\label{different}.
\end{eqnarray}
 Note that conditions \eqref{buncond}--\eqref{different} mean that ${\cal S}(\la)$ is congruent to ${\cal K}_{a'}(\la)$, for some eigenvalues $\mu_1,\hdots,\mu_{r-2a'}$ in \eqref{max}, different from each other. Moreover, \eqref{buncond} implies that
 \begin{equation}\label{rankcond}
 \rank {\cal S}=r.
 \end{equation}
  Notice, also, that $\rank {\cal C}(\la)^\top\leq a$, so condition \eqref{rankcond} implies that $\rank {\cal R}=r-2a$ and $\rank {\cal C}=a$. Then, ${\cal R}(\la)$ is a regular pencil with $r-2a$ eigenvalues (counting multiplicities). Let us denote these eigenvalues by $\widetilde \mu_1,\hdots,\widetilde \mu_{r-2a}$. By \eqref{buncond}, the pencil ${\cal S}(\la)$ has more than $r-2a$ eigenvalues (counting multiplicities). If $\rank {\cal C}(\la_0)=a$, for all $\la_0\in\CC\cup\{\infty\}$ (where ${\cal C}(\infty)$ is the leading term of ${\cal C}(\la)$) then $\rank {\cal S}(\mu)=\rank {\cal S}=r$ for all $\mu\neq \widetilde \mu_i$ ($i=1,\hdots,r-2a$), which means that ${\cal S}(\la)$ has only $r-2a$ eigenvalues. Therefore, there must be some $\la_0\in\CC\cup\{\infty\}$ such that $\rank {\cal C}(\la_0)<a$. In particular, such $\la_0$ is an eigenvalue of ${\cal S}(\la)$, since the number of linearly independent rows of ${\cal S}(\la_0)$ is less than $r$. Now, we consider the following two cases:

  \medskip

  \begin{itemize}
  \item[Case 1,] $\rank{\cal S}(\la_0)=\rank{\cal H}(\la_0)\leq r-2$: In this case, $\la_0$ is an eigenvalue of ${\cal S}(\la)$ with geometric multiplicity at least two, contradicting \eqref{different}.

 \item[Case 2,] $\rank{\cal S}(\la_0)=\rank{\cal H}(\la_0)= r-1$: We are going to see that, in this case, $\la_0$ is an eigenvalue of ${\cal S}(\la)$ with algebraic multiplicity at least 2, which is in contradiction with \eqref{different} as well. For this, we will prove that all $r\times r$ non-identically zero minors of ${\cal H}(\la)$ have $(\la-\la_0)^2$ as a factor. In order for an $r\times r$ submatrix of ${\cal H}(\la)$ to have non-identically zero determinant, it must contain fewer than $a+1$ columns among the last $n-r+a$ columns of ${\cal H}(\la)$ (namely, those corresponding to ${\cal C}(\la)$), and fewer than $a+1$ rows among the last $n-r+a$ rows of ${\cal H}(\la)$ (namely, those corresponding to ${\cal C}(\la)^\top$). This is because any set of $a+1$ columns among the last $n-r+a$ columns of ${\cal H}(\la)$ is linearly dependent, and the same for the last $n-r+a$ rows. As a consequence, any $r\times r$ non-identically zero minor, $M(\la)$, of ${\cal H}(\la)$ is of the form:
 $$
 M(\la):=\det\left[\begin{array}{ccc}{\cal A}(\la)&{\cal B}(\la)&{\cal C}_1(\la)\\{\cal B}^\top(\la)&{\cal R}(\la)&0\\{\cal C}_2(\la)^\top&0&0\end{array}\right],
 $$
 where ${\cal C}_1(\la)$ (respectively, ${\cal C}_2(\la)^\top$) is a submatrix of ${\cal C}(\la)$ (resp., ${\cal C}(\la)^\top$) with size $a\times a$. Therefore, $M(\la)=\pm\det {\cal R}(\la)\cdot\det{\cal C}_1(\la)\cdot\det{\cal C}_2(\la)^\top$. Since $\rank{\cal C}_1(\la_0)<a$ and $\rank{\cal C}_2(\la_0)^\top<a$, the binomial $(\la-\la_0)$ is a factor of both $\det {\cal C}_1(\la)$ and $\det{\cal C}_2(\la)^\top$, so $(\la-\la_0)^2$ is a factor of $M(\la)$.

 Therefore, the gcd of all $r\times r$ non-identically zero minors of ${\cal H}(\la)$ is a multiple of $(\la-\la_0)^2$. This implies (see, for instance, \cite[p. 141]{Gant59}) that the algebraic multiplicity of $\la_0$ as an eigenvalue of ${\cal H}(\la)$, and so of ${\cal S}(\la)$, is at least $2$.
 \end{itemize}
  \end{proof}

It is natural to ask about the generic eigenstructures of Hermitian pencils, instead of the symmetric ones. Despite the close similarities between both structures (Hermitian and symmetric), the Hermitian case presents a striking difference with respect to the symmetric case. More precisely, the set of $n\times n$ singular Hermitian matrix pencils (that is, the set of $n\times n$ Hermitian pencils with rank at most $r=n-1$) is irreducible {\rm\cite[Th. 4]{Wate84}}. However, Theorem {\rm\ref{th:improvedDeDo}} says that $\pen_{n\times n}^s(n-1)$ is the union of $\lfloor\frac{n-1}{2}\rfloor+1$ different closed subsets, so $\pen_{n\times n}^s(n-1)$ has, at least, $\lfloor\frac{n-1}{2}\rfloor+1$ irreducible components.  Actually, it has exactly $\lfloor\frac{n-1}{2}\rfloor+1$ irreducible components \cite[Th. 3]{Wate84}. Nonetheless, it is natural to ask whether the arguments used in this paper for symmetric pencils can be adapted to the Hermitian case.

Hermitian pencils are defined in the same way as symmetric pencils just replacing the transpose ($\top$) by the conjugate transpose ($*$) in the pencil coefficients. Therefore, a natural and straightforward attempt to adapt the proof of Theorem \ref{th:improvedDeDo} to the set of Hermitian pencils consists of just replacing $\top$ by $*$ and following the same arguments. There is a first difficulty in this approach, namely that the canonical form under conjugate transpose congruence of Hermitian pencils contains some additional structure due to the presence of pairs of complex conjugate eigenvalues and real eigenvalues with some associated {\em sign characteristic} \cite[Th. 6.1]{LaRo05}. A second difficulty comes from the fact that, after replacing $\top$ by $*$ in \eqref{h}, and adapting the subsequent arguments in the proof of Theorem \ref{th:improvedDeDo},  we would conclude that $\rank {\cal C}(\la_0)^*<a $ whenever $\rank {\cal C}(\la_0)<a $. However, this does not guarantee that $\la_0$ is a multiple eigenvalue of \eqref{h} (with $*$ instead of $\top$). Note that, if ${\cal C}(\la)=\la C_1+C_0$, then ${\cal C}(\la_0)^*=\overline\la_0C_1^*+C_0^*$ consists of evaluating not $\la_0$ but its complex conjugate in ${\cal C}^*(\la):=\la C_1^*+C_0^*$ (the key fact here is that the pencil ${\cal C}^*(\la)$ is not the same as ${\cal C}(\la)^*=\overline\la C_1^*+C_0^*$, unlike what happens with ${\cal C}^\top(\la)$ and ${\cal C}(\la)^\top$). Therefore, the Hermitian case deserves a different treatment.

\section{Irreducible components of $\pen_{n\times n}^s(r)$}\label{sec:irredcomp}

The set $\pen_{n\times n}^s(r)$ is an algebraic (namely, Zariski closed) subset of $\pen_{n\times n}$, when identified with $\CC^{2n^2}$, since it is defined as the set of common zeroes of several multivariable polynomials. These polynomials correspond to imposing both the conditions on the symmetric structure on $\la A+B$ (namely, $a_{ij}-a_{ji}=0$ and $b_{ij}-b_{ji}=0$, for $i<j$, where $A=[a_{ij}]$ and $B=[b_{ij}]$) and the low-rank conditions, $p_\ell(a_{ij},b_{ij})=0$, where the polynomials $p_\ell$ are the coefficients of all $(r+1)\times (r+1)$ minors of $\la A+B$ (these minors are polynomial in $\la$). It is known \cite[Th. 3]{Wate84} that, for $r=n-1$, this set has $\lfloor\frac{n+1}{2}\rfloor=\lfloor\frac{r}{2}\rfloor+1$ irreducible components. Theorem
\ref{irred_th} generalizes this result to any $r<n$ and provides a description of the irreducible components. Before proving it, we need to see that the closures of $\bun^c({\cal K}_a)$ in both the Zariski and the Euclidean topology coincide. This is an immediate consequence of \cite[Th. 3.1.6.1]{landsberg17} and the following two facts:
\begin{itemize}
\item[(F1)] $\bun^c({\cal K}_a)$ is a Zariski open set in its (Zariski) closure.
\item[(F2)] The Zariski closure of $\bun^c({\cal K}_{a})$ is an irreducible set.

\end{itemize}

To see (F1) above, note that, for any fixed set of eigenvalues, $\mu_1,\hdots,\mu_{r-2a}$, the orbit $\orb^c({\cal K}_a)$ is an open set in its closure \cite[p. 60]{Hump75}. Then, the same happens for $\bun^c({\cal K}_a)$, since it is the union of these orbits. Claim (F2) is stated and proved in the following theorem.

\begin{theorem}\label{th:irredbun}
The (Zariski) closure of the bundle $\bun^c({\cal K}_{a})$ is an irreducible set.
\end{theorem}
\begin{proof}
The proof follows similar arguments to the ones in the proof of Lemma~3.4 in \cite{DeDo08}. The closure of $\bun^c({\cal K}_{a})$ throughout the proof is considered in the Zariski topology.

Let us consider the following map:
$$
\begin{array}{cccc}
\Phi:&\gln\times (\CC\cup\{\infty\})^{r-2a}&\rightarrow&\pen_{n \times n}\\
&(P;\mu_1,\hdots,\mu_{r-2a})&\mapsto&P{\cal K}_a(\la)P^\top,
\end{array}
$$
where ${\cal K}_a(\la)$ has the eigenvalues $\mu_1,\hdots,\mu_{r-2a}$ as in \eqref{max}. Then it is straightforward to see that:
\begin{itemize}
\item[(a)] $\Phi$ is a polynomial map.

\item[(b)] $\Phi(\gln\times \CC^{r-2a})=\bun^c({\cal K}_{a})$.
\end{itemize}

In general, for a given subset $A$ of a topological space $(X,{\cal T}_X)$, and a continuous map between topological spaces $\phi:(X,{\cal T}_X)\rightarrow (Y,{\cal T}_Y)$, the following inclusion holds:
\begin{equation}\label{closures}
\phi\left(\overline A\right)\subseteq \overline{\phi(A)},
\end{equation}
where the closures are taken in ${\cal T}_X$ and ${\cal T}_Y$, respectively. To see this, just note that $A\subseteq\phi^{-1}\left(\overline{\phi(A)}\right)$ and, since $\phi$ is continuous, the set $\phi^{-1}\left(\overline{\phi(A)}\right)$ is closed in ${\cal T}_X$. Therefore
$
\overline A\subseteq \phi^{-1}\left(\overline{\phi(A)}\right),
$
and this implies \eqref{closures}.

Since $\gln$ is dense in $\CC^{n\times n}$ (both in the Euclidean and the Zariski topology), it follows that $\overline{\gln\times\CC^{r-2a}}=\CC^{n\times n}\times\CC^{r-2a}$. Since $\Phi$ is a polynomial function, it is continuous when considering in both $\gln\times \CC^{r-2a}$ and $\pen_{n \times n}$ the Zariski topology (where $\pen_{n \times n}$ is identified with $\CC^{2n^2}$). As a consequence of \eqref{closures} and (b) above,
$
\overline{\Phi(\CC^{n\times n}\times\CC^{r-2a})}=\overline{\Phi\left(\overline{\gln\times \CC^{r-2a}}\right)}=\overline{\bun^c}({\cal K}_a).
$
Now the result follows from the fact that the Zariski closure of the image of any $\CC^m$ by a polynomial map is irreducible (see \cite[p. 228]{Wate84}).
\end{proof}

As a consequence of Theorem \ref{th:irredbun} and the precedent comments, the Euclidean and Zariski closures of $\bun^c({\cal K}_a)$ coincide. Now claims (ii) and (iii) in Theorem \ref{th:improvedDeDo}, together with Theorem \ref{th:irredbun}, directly imply the following result.

\begin{theorem}\label{irred_th}
The set of $n\times n$ symmetric matrix pencils with rank at most $r<n$ is an algebraic set having $\lfloor\frac{r}{2}\rfloor+1$ irreducible components, which are the closure (either in the Zariski or the Euclidean topology) of the congruence bundles $\bun^c({\cal K}_a)$, with ${\cal K}_a(\la)$ as in \eqref{max}, for $a=0,1,\hdots,\lfloor\frac{r}{2}\rfloor$.
\end{theorem}

\section{Codimensions of the generic bundles}
\label{sec:codim}

For an $n\times n$ symmetric matrix pencil $\lambda A + B$, the {\it dimension} of $\orb^c (\lambda A + B)$ is defined to be the dimension of the tangent space to this orbit
\begin{equation*}\label{taneqsp}
\tsp_{\lambda A + B}^c:=\{\lambda (X^\top A+AX) + (X^\top B + BX):
X\in{\mathbb
C}^{n\times n}\}
\end{equation*}
at the point $\lambda A + B$.
Define the normal space to $\orb^c (\lambda A + B)$ at the point $\lambda A + B$ to be the orthogonal complement to $\tsp_{\lambda A + B}^c$ with respect to the Frobenius inner product $<\la A+B,\la C+D>:=\tr(AC^*+BD^*)$. The {\it codimension} of the congruence orbit of $\lambda A + B$ is the dimension of the normal space. This codimension is equal to $n(n+1)$ minus the dimension of the congruence orbit of $\lambda A + B$.
Explicit expressions for the codimensions of congruence orbits of symmetric pencils in $\pen_{n\times n}^{s}$ are derived in \cite{DmKS14} and implemented in the MCS (Matrix Canonical Structure) Toolbox \cite{DmJK13, Joha06}.
Define the codimension of $\bun^c(\lambda A + B)$ as:
\begin{equation*} \label{buncodim}
    \cod  \bun^c(\lambda A + B):=
    \cod  \orb^c (\lambda A + B) - \ \#
    \left\{ \text{distinct eigenvalues of } {\lambda A + B} \right\}.
\end{equation*}
Using \cite[Theorem~2.3]{DmKS14}, see also \cite[Theorem~2.7]{DmJK13}, we obtain $\cod \orb^c({\cal K}_{a})$ for the symmetric pencils ${\cal K}_{a} (\lambda)$ in \eqref{max} as follows:
\begin{equation*} \label{codimcomporb}
\begin{aligned}
\cod \orb^c({\cal K}_{a}) &\phantom{a} = r-2a+2(\alpha +2)s + 2(\alpha +1)(n-r-s) + (n-r)(r-2a)  \\ 
&\phantom{a}  + (n-r-s)(n-r-s-1)(2\alpha +2)/2 + s(s-1)(2(\alpha+1) +2)/2 \\
&\phantom{a}  + s(n-r-s)(2(\alpha+1) +1)\\
&\phantom{a} = r-2a+2(\alpha(n-r) +n -r+s) + (n-r)(r-2a)\\
&\phantom{a} + (n-r-s)(n-r-s-1)(\alpha +1) + s(s-1)(\alpha+1)  + s(s-1) \\
&\phantom{a} + 2s(n-r-s)(\alpha+1) + s(n-r-s) \\
&\phantom{a} = r-2a + 2(a + n - r) + (n-r)(r-2a)\\
&\phantom{a} + (\alpha +1) ((n-r-s)(n-r-1) + s(n-r-1)) + s(n-r-1) \\
&\phantom{a} = r-2a + 2(a + n - r) + (n-r)(r-2a)\\
&\phantom{a} + (\alpha +1) (n-r)(n-r-1) + s(n-r-1) \\
&\phantom{a} = n + (n-r)(r-2a+1) + (a+n-r)(n-r-1) \\
&\phantom{a} = (n-r)(n-a+1)+r -a =(n-a)(n-r+1).
\end{aligned}
\end{equation*}
Therefore for the generic bundles we have:
\begin{equation*} \label{codimcompbun}
\begin{aligned}
\cod \bun^c({\cal K}_{a}) &= (n-a)(n-r+1) - r + 2a = n(n-r+1) - r - a(n-r-1)  \\
& = (n+1)(n-r)-a(n-r-1).
\end{aligned}
\end{equation*}

This shows, in particular, that all generic bundles in Theorem \ref{th:improvedDeDo} have different codimension if $r<n-1$ and, in this case, the one with largest $a$ has the smallest codimension or, equivalently, the largest dimension.

\section{Future work}
\label{sec:future}
Further natural development of the results of this paper would be a description of the generic eigenstructures for symmetric matrix polynomials of bounded grade and rank. Such an extension requires a result on the existence of a symmetric matrix polynomial with a prescribed complete eigenstructure, similar to the results obtained in \cite{DeDV15,Dmyt17} for general and skew-symmetric matrix polynomials.

 Another natural closely related open problem is to provide a description of the set of $n\times n$ Hermitian pencils with bounded rank analogous to the one in Theorem~\ref{th:improvedDeDo}. The Hermitian structure requires a separate analysis since, as we have mentioned right after the proof of Theorem \ref{th:improvedDeDo}, some of the arguments used in this paper for symmetric pencils are not valid for Hermitian ones. Moreover, the result for Hermitian pencils will be probably very different to the one we have obtained in this paper for symmetric pencils.

The results of this paper show an important difference between general and symmetric perturbations of symmetric matrix pencils, see Example \ref{introex}. This is crucial for developing the stratification theory, i.e. explaining all the possible changes of complete eigenstructures under structure-preserving infinitesimally small perturbations, for symmetric matrix pencils. Currently stratification theory is developed for general, skew-symmetric, and state-space pencils \cite{DmJK17,DmKa14,EdEK99} but the theory for symmetric matrix pencils is still missing. This paper, together with \cite{Dmyt11,DmKS14} can be seen as important steps towards the stratification theory for symmetric matrix pencils. The problem of stratification of symmetric matrix polynomials is also open and will be a natural extension of the problems above.

\bigskip

\noindent{\bf Acknowldegments.} This work has been partially supported by the Ministerio de Econom\'{i}a y Competitividad of Spain
through grant MTM2015-65798-P, and by the Ministerio de Ciencia, Innovaci\'{o}n y Universidades of Spain through grant MTM2017--90682--REDT (Fernando De Ter\'{a}n and Froil\'{a}n M. Dopico).

\bibliographystyle{siamplain}

\end{document}